\documentclass[11pt,reqno]{amsart}
\usepackage[tmargin=1in,bmargin=1in,rmargin=1in,lmargin=1in]{geometry}

\usepackage[breaklinks=true]{hyperref}
\usepackage{soul}
\usepackage{setspace}
\usepackage{caption,xcolor}
\usepackage{amsmath,amsthm,amsfonts,amssymb,tocvsec2}
\usepackage{cancel}
\usepackage{algorithm}
\usepackage{algpseudocode}
\usepackage{enumerate}

\theoremstyle{plain}
\newtheorem{theorem}{Theorem}[section]

\newtheorem{lemma}[theorem]{Lemma}

\newtheorem{cor}[theorem]{Corollary}
\newtheorem{conjecture}{Conjecture}

\theoremstyle{definition}
\newtheorem{definition}[theorem]{Definition}

\newtheorem{rem}[theorem]{Remark}
\newtheorem{example}[theorem]{Example}

\numberwithin{equation}{section}

\DeclareMathOperator{\GCD}{GCD}

\newcommand{\mb}{\mathbb}
\newcommand{\mc}{\mathcal}

\begin{document}
\title[Positivity of GCD tensors and their determinants]{Positivity of GCD tensors and their determinants}

\author{Projesh Nath Choudhury and Krushnachandra Panigrahy}
\address[P.N.~Choudhury]{Department of Mathematics, Indian Institute of Technology Gandhinagar, Gujarat 382355, India}
\email{\tt projeshnc@iitgn.ac.in}

\address[K.~Panigrahy]{Department of Mathematics, SRM University-AP, Amaravati, 522502, India}
\email{\tt kcp.224@gmail.com, krushnachandra.p@srmap.edu.in}

\date{\today}

\begin{abstract}
Let $S=\{s_{1},s_{2},\ldots,s_{n}\}$ be an ordered set of $n$ distinct positive integers.  The $m$th-order $n$-dimensional tensor $\mc{T}_{[S]}=(t_{i_{1}i_{2}\ldots i_{m}}),$ where $t_{i_{1}i_{2}\ldots i_{m}}=\GCD(s_{i_{1}},s_{i_{2}},\ldots,s_{i_{m}}),$ the greatest common divisor (GCD) of $s_{i_{1}},s_{i_{2}},\ldots,$ and $s_{i_{m}}$ is called the GCD tensor on $S$. The earliest result on GCD tensors goes back to Smith [\textit{Proc. Lond. Math. Soc.}, 1976], who computed the determinant of GCD matrix on $S=\{1,2,\ldots,n\}$ using the Euler's totient function, followed by Beslin--Ligh [\textit{Linear Algebra Appl.}, 1989] who showed all GCD matrices are positive definite. In this note, we study the positivity of higher-order tensors in the $k$-mode product. We show that all GCD tensors are strongly completely positive (CP).  We then show that GCD tensors are infinite divisible. In fact, we prove that for every {positive} real number $r,$ the tensor $\mc{T}_{[S]}^{\circ r}=(t^{r}_{i_{1}i_{2}\ldots i_{m}})$ is strongly CP. Finally, we obtain an interesting decomposition of GCD tensors using Euler's totient function $\Phi$. Using this decomposition, we show that the determinant (also called hyperdeterminant) of the $m$th-order GCD tensor $\mc{T}_{[S]}$ on a factor-closed set $S=\{s_1,\dots,s_n\}$ is $\prod\limits_{i=1}^{n} \Phi(s_{i})^{(m-1)^{(n-1)}}$.
\end{abstract}



\maketitle

\section{Introduction and main results}

{\em Given integers $m$ and $n,$ we define $\mb{R}^{[m,n]}:$ the set of $m$th-order $n$-dimensional tensors; $\mb{R}_{+}^{n}:$ the set of all nonnegative vectors in $\mb{R}^{n}.$ For $\mc{A}\in\mb{R}^{I_{1}\times I_{2}\times \cdots \times I_{m}}$ and $B\in\mb{R}^{J\times I_{k}}$ where $1\leq k\leq m$ is an integer, the $k$-mode product is denoted by $\mc{A}\times_k B$ and defined elementwise as
\begin{equation*}
	(\mc{A}\times_{k}B)_{i_{1}i_{2}\ldots i_{k-1} j i_{k+1} \ldots i_{m}}=\sum_{i_{k}=1}^{I_{k}}\mc{A}_{i_{1}i_{2}\ldots i_{k-1} i_{k} i_{k+1} \ldots i_{m}}B_{ji_{k}}.
\end{equation*}}
\smallskip

Given integers $m$ and $n$, a tensor $\mc{A}=(a_{i_{1}i_{2}\ldots i_{m}})\in\mb{R}^{[m,n]}$ is said to be {\it symmetric} if its entries are invariant under any permutation of their indices. A symmetric tensor $\mc{A}\in\mb{R}^{[m,n]}$ is positive semidefinite (positive definite) if $\mc{A}x^{m}\geq 0 ~(\mc{A}x^{m}> 0)$ for all $0\neq x\in \mb{R}^{n}$. Clearly, positive semidefiniteness vanishes when the order of the tensor is odd. 
Positive definite and semidefinite even-order tensors have rich applications in several branches of science and engineering including stochastic processes, polynomial theory, spectral hypergraph theory, automatic control, magnetic resonance imaging \cite{benson2017, chen2012, chen2013, draisma2024, hu2012, huq2012, hu2015, jiang2016,  qi2010}. 

In 2013, Hillar--Lim \cite{hillar2013} showed that verifying the positive definiteness of a tensor is an NP-hard problem. Since then, the positive semidefiniteness of various structured tensors has been studied in the literature. Qi--Song \cite{qi2014a} showed that the class of an even-order $B$- ($B_{0}$-) tensor is positive definite (positive semidefinite). In \cite{li2014}, Li--Wang--Zhao--Zhu--Li proved that an even-order symmetric strong $H$-tensor with positive diagonal entries is a positive definite tensor. Furthermore, in 2015, Kannan--Shaked-Monderer--Berman \cite{kannan2015} showed that a symmetric $H$-tensor with nonnegative diagonal entries is positive semidefinite. In the same year, Qi \cite{qi2015} proved that an even-order strong Hankel tensor is positive semi-definite (PSD). Chen--Qi \cite{haibin2015} showed that an even-order symmetric Cauchy tensor is PSD if and only if its generating vector is positive. {Recently, Cui--Qi--Chen \cite{cui2025} proved that the even-order Pascal tensors are positive definite and the odd-order Pascal tensors are strongly completely positive. They also showed that the determinant of the $m$th-order two-dimensional symmetric Pascal tensor is $[(m-1)!]^m.$}

In this work, we explore the positivity of greatest common divisor tensors (GCD tensors). We need some basic definitions and notations to continue our discussion.
\begin{definition}\label{def:gcdtnsr2}
Throughout the manuscript, let $m, n \geq 2$ be integers.
\begin{enumerate}[(i)]  
    \item Let $S=\{s_{1},s_{2},\ldots,s_{n}\}$ be a set of $n$ distinct positive integers. Then the $m$th-order $n$-dimensional tensor $\mc{T}_{[S]}=(t_{i_{1}i_{2}\ldots i_{m}})\in\mb{R}^{[m,n]},$ such that 
	\[t_{i_{1}i_{2}\ldots i_{m}}=\GCD(s_{i_{1}},s_{i_{2}},\ldots,s_{i_{m}})\]
	is called the greatest common divisor (GCD) tensor on $S$.
    \item Euler's totient function $\Phi(n)$ counts the number of positive integers less than or equal to $n$ that are coprime to $n$.
\end{enumerate}
	
\end{definition}

The GCD matrices (second-order GCD tensors) were first studied by H.J.S. Smith in 1876. In \cite{smith1876}, he showed that for $S=\{1,2,\ldots,n\}$ the determinant of the GCD matrix $\mc{T}_{[S]}=(t_{i_{1}i_{2}})\in\mb{R}^{[2,n]}$ is $\Phi(1)\cdot \Phi(2)\cdots\Phi(n)$. A classical 1990 result of Beslin--Ligh \cite{beslin1989} showed that GCD matrices $\mc{T}_{[S]}\in\mb{R}^{[2,n]}$ are positive definite for any set $S=\{s_{1},s_{2},\ldots,s_{n}\}$ of distinct positive integers. The GCD matrices and their generalizations have since attracted the attention of number theorists and linear algebraists  \cite{ercan2017,bhat1991,haukkanen2018,haukkanen1997,lindstrom1969,mattila2014,wilf1968}. Recently, Guillot--Wu \cite{guillot2019} studied the total nonnegativity of GCD matrices.

In this short note, we study the positive definiteness of higher-order GCD tensors. To state our main results, we need the following definitions.
\begin{definition}\cite{qi2014}\label{dfn:cptnsr} Let $\mc{A}=(a_{i_{1}i_{2}\ldots i_{m}})\in\mb{R}^{[m,n]}$ be a symmetric tensor. Then
 $\mc{A}$ is said to be \textit{completely positive} (CP) if there exist $v_{1},v_{2},\ldots,v_{r}\in\mb{R}^{n}_{+}$ such that $\mc{A}=(v_{1})^{\otimes m} + (v_{2})^{\otimes m} + \cdots +(v_{r})^{\otimes m},$ where $\otimes$ denotes the outer product, and $(\cdot)^{\otimes m}=\underbrace{(\cdot)\otimes (\cdot)\otimes \cdots\otimes(\cdot)}_{m~{\rm times}}$. Moreover, if ${\rm span}\{v_{1},v_{2},\ldots,v_{r}\}=\mb{R}^{n}$, then $\mc{A}$ is called a \textit{strongly completely positive} tensor. 	
\end{definition}

{Our first main result shows that every GCD tensor is strongly CP.}
\begin{theorem}\label{thm:gcdcpt}
	Let $S=\{s_{1},s_{2},\ldots, s_{n}\}$ be a set of $n$ distinct positive integers. 
    Then the $m$th-order $n$-dimensional GCD tensor $\mc{T}_{[S]}$ on $S$ is a strongly completely positive tensor.
\end{theorem}
As an immediate consequence of this result, we provide an algorithm for deriving the strongly CP decomposition of GCD tensors.

We next study the infinite divisibility of the GCD tensor. 
\begin{definition}
Let $\mc{A}=(a_{i_{1}i_{2}\ldots i_{m}})\in\mb{R}^{[m,n]}$ be a positive semideﬁnite tensor with $a_{i_{1}i_{2}\ldots i_{m}}\geq 0$ for all $i_{1},i_{2},\ldots, i_{m}$. Then $\mc{A}$ is called inﬁnitely divisible if for every nonnegative real number $r$, the fractional Hadamard power $\mc{A}^{\circ r}:=(a_{i_{1}i_{2}\ldots i_{m}}^{r})$ is positive semideﬁnite. 
\end{definition}
Infinite divisible matrices and kernels are extensively studied in the literature. These matrices and kernels attracted renewed attention in 1966 when Loewner \cite{loewner1966} showed that under certain conditions the Green's function for Laplace's equation is an infinitely divisible kernel.  In \cite{luo2016}, Luo--Qi showed that the positive Cauchy tensor, symmetric Pascal tensor, and several classes of mean tensors are infinitely divisible -- more strongly, the fractional Hadamard powers of these classes of tensors are strongly completely positive. We obtain a similar result for GCD tensors.
\begin{theorem}\label{thrmfractpwr}
Let $S=\{s_{1},s_{2},\ldots, s_{n}\}$ be a set of distinct positive integers. Let $\mc{T}_{[S]}$ be the $m$th-order $n$-dimensional GCD tensor on $S$. Then $\mc{T}_{[S]}^{\circ r}$ is strongly completely positive for every {positive real number $r$. 
}
\end{theorem}
\begin{rem}
	{If $r=0$ in the preceding theorem then $\mc{T}^{\circ r}_{[S]}$ is a tensor with every entry $1,$ and hence $\mc{T}^{\circ 0}_{[S]}=v^{\otimes m},$ where $v=(1,1,\ldots,1)^{T}\in\mb{R}^{n}$. Thus $\mc{T}^{\circ 0}_{[S]}$ is a completely positive tensor.}
\end{rem}

We now turn our attention to the determinant of the GCD tensor. The study of hyperdeterminants of tensors was initiated by Cayley \cite{cayley1845} in 1845, and was revived a century later in 1994 by  Gelfand--Kapranov--Zelevinsky \cite{gelfand2008}. In 2005, Qi \cite{qi2005}  introduced the notion of symmetric hyperdeterminant of symmetric tensors to study the eigenvalues of symmetric tensors.  In 2013, Hu--Huang--Ling--Qi \cite{hu2013}  defined the determinant of all $m$th-order $n$-dimensional tensors using the resultant of multivariate polynomials. To define the determinant of a tensor, we need the following notation. Let $\mc{A}\in\mb{R}^{[m,n]}$ and $x=(x_1,{x_{2},}\ldots,x_n)^T\in\mb{R}^{n}$. Then
\[\mc{A}x^{m-1}
=\mc{A}\times_{2}x^{T}\times_{3}\cdots\times_{m}x^{T}=\left(\sum_{i_{2},\ldots,i_{m}=1}^{n}a_{i_{1}i_{2}\ldots i_{m}}x_{i_{2}}\cdots x_{i_{m}}\right)\in \mb{R}^{n}.\]  Now define the multivariate homogeneous polynomials of degree $m-1$
\[F_{k}(x_{1},x_{2},\ldots,x_{n}):=\sum_{i_{2},\ldots,i_{m}=1}^{n}a_{ki_{2}\ldots i_{m}}x_{i_{2}}\cdots x_{i_{m}} \hbox{ for } 1\leq k\leq n.\]
Equipped with these $n$ polynomials, we are now ready to define the determinant of a tensor.
\begin{definition}\label{defndet}
Let $\mc{A}\in\mb{R}^{[m,n]}$ with $m\geq 2$. Let $F_1,{F_{2},}\ldots,F_n\in \mathbb{R}[x_1,{x_{2},}\ldots,x_n]$ be homogeneous polynomials of degree $m-1$ as defined above. Then the determinant of $\mc{A}$, denoted by $\det(\mc{A})$, is the resultant of the ordered set of polynomials $F_1,{F_{2},}\ldots, F_n$.  
\end{definition}
From the above definition, it is easy to verify that the tensor determinant generalizes the well-known matrix determinant and symmetric hyperdeterminant studied by Qi \cite{qi2005}.

We now obtain the determinant of the GCD tensor. To state our next main result, we first recall the following basic definition.
\begin{definition}
 A set $S$ of positive integers is said to be {\it factor-closed} (FC), if whenever $d$ divides $s\in S$, then $d\in S.$ 
\end{definition}

Our final main result gives the determinant of the GCD tensor $\mc{T}_{[S]}$ on a factor-closed $S$ in terms of the Euler totient function.

\begin{theorem}\label{detgcd}
Let $S=\{s_{1},s_{2},\ldots, s_{n}\}$ be a factor-closed set of distinct positive integers. Let $\mc{T}_{[S]}$ be the $m$th-order $n$-dimensional GCD tensor on $S$.  Then 
\[\det{(\mc{T}_{[S]})}=\prod_{i=1}^{n} \Phi(s_{i})^{(m-1)^{(n-1)}}.\]
\end{theorem}

We now explain the organization of the paper. In the next section, we present the proofs of Theorems \ref{thm:gcdcpt} and \ref{thrmfractpwr}, and provide an algorithm for computing the strongly CP decomposition of GCD tensors. Section \ref{sec:det} is dedicated to studying the determinant of the GCD tensor. In this section, we establish our final main result, Theorem \ref{detgcd}, by developing a new factorization of the GCD tensor. Furthermore, we discuss the determinants of the GCD tensor on a generalization of the factor-closed set. The paper concludes by extending the concept of the GCD tensor on the meet semi-lattice in Section \ref{sec:poset}.

On a more philosophical note, CP tensors and strongly CP tensors have rich applications in diverse areas, including statistics, computer vision, exploratory multiway data analysis, blind source separation, and higher-degree polynomial optimization \cite{cichocki2009,fan2014,qi2014,shashua2005}. To understand and explore the applications of any class of tensors, we need a vast collection of examples. The first main result of this paper provides a new class of examples of strongly CP tensors. This class of examples is easy to construct and arises from any set of distinct positive integers and their GCDs. Moreover, our second main result shows that from any GCD tensor, one can provide uncountably many strongly CP tensors using fractional Hadamard powers. We hope this class of examples will be a stem for a wide range of applications.

\section{Strongly completely positivity of GCD tensors}\label{sec:cp}
In this section, we prove two of our main results -- Theorems \ref{thm:gcdcpt} and \ref{thrmfractpwr}, and outline the additional work required to complete their proofs.  We begin by proving Theorem \ref{thm:gcdcpt}. This requires a basic lemma due to Gauss for Euler's totient function.

\begin{lemma}\label{gauss}
For any positive integer $k$, $\sum\limits_{d\mid k}\Phi(d)=k$.
\end{lemma}

\begin{proof}[Proof of Theorem \ref{thm:gcdcpt}]
We begin by showing that $\mc{T}_{[S]}$ is completely positive. Let $F=\{f_{1},f_{2},\ldots,f_{l}\}$ be a factor-closed set containing $S$~ (so $l\geq n$). Define the matrix $E:=(e_{ij})\in\mb{R}^{n\times l}$ such that
\begin{equation}\label{matrixE}
	e_{ij}:=\begin{cases}
		1,&\mathrm{if }~f_{j}|s_{i},\\
		0,&\mathrm{otherwise.}
	\end{cases}
\end{equation}
We claim that the $m$th-order $n$-dimensional GCD tensor $\mc{T}_{[S]}$ can be decomposed as
\begin{equation}\label{eqn:mre}
	\mc{T}_{[S]}=\sum_{k=1}^{l}\Phi(f_{k})(E_{k})^{\otimes m},
\end{equation}
where $E_{k}$ denotes the $k$th column of $E$. Fix $1\leq i_1,{i_{2},}\ldots,i_m\leq n$.  To prove the equality \eqref{eqn:mre} it suffices to show that the $(i_{1},i_{2},\ldots,i_{m})$th element of the summation  $\sum\limits_{k=1}^{l}\Phi(f_{k})(E_{k})^{\otimes m}$ agrees with the $(i_{1},i_{2},\ldots,i_{m})$th element of  $\mc{T}_{[S]}$.  But notice that for $1\leq k\leq l$, the $(i_{1},i_{2},\ldots,i_{m})$th element of $(E_{k})^{\otimes m}$ is $e_{i_{1}k} e_{i_{2}k}\cdots e_{i_{m}k}$ which is equal to $1$ if $f_{k}|s_{i_{1}}, f_{k}|s_{i_{2}}, \ldots ,f_{k}|s_{i_{m}}$, otherwise $0$. Thus,
\begin{eqnarray*}
	\left(\sum_{k=1}^{l}\Phi(f_{k})(E_{k})^{\otimes m}\right)_{i_{1}{i_{2}}\ldots i_{m}}&=&\sum_{f_{k}\mid s_{i_{1}},f_{k}\mid s_{i_{2}},\ldots,f_{k}\mid s_{i_{m}}}\Phi(f_{k})\\
	&=&\sum_{f_{k}\mid \GCD(s_{i_{1}},s_{i_{2}},\ldots,s_{i_{m}})}\Phi(f_{k})\\	&=&\GCD(s_{i_{1}},s_{i_{2}},\ldots,s_{i_{m}}).
\end{eqnarray*}
The last equality follows from Lemma \ref{gauss}. Thus $\mc{T}_{[S]}$ is completely positive, since $\Phi(f_{k})>0$ for all $1\leq k \leq n$. 

We next show that ${\rm span} \{E_1,{E_{2},}\ldots,E_l\}=\mathbb{R}^n$. Let $S'=\{s_{r_1},{s_{r_2}},\ldots, s_{r_n}\}$ be the rearrangement of the elements of $S$ such that $s_{r_1}<s_{r_2}<\cdots< s_{r_n}$ and let $F'=\{f_{r_1},f_{r_2},\ldots,f_{r_l}\}$ be obtained by rearranging the elements of  $F$ such that $f_{r_1}=s_{r_1}, f_{r_2}=s_{r_2},\ldots, f_{r_n}=s_{r_n}$. Define the matrix $E':=(e'_{ij})\in\mb{R}^{n\times l}$ elementwise as follows:
\begin{equation}
	e'_{ij}:=\begin{cases}
		1,&\text{if } f_{r_j} \text{ divides }s_{r_i},\\
		0,&\text{otherwise.}
	\end{cases}
\end{equation}
Then there exist permutation matrices $P\in \mathbb{R}^{n \times n}$ and $Q\in \mathbb{R}^{l \times l}$ such that $E'=PEQ$. Thus ${\rm rank}(E')={\rm rank}(E).$ To show that the columns of $E$ span $\mb{R}^{n}$, it is sufficient to show that $E'$ has rank $n$. Now, on writing the first $n (\leq {l})$ columns of the matrix $E'$ explicitly we have
\begin{equation*}
	E'=\begin{pmatrix}
		1&0&0&\cdots&0&0&*&\cdots&*\\
		*&1&0&\cdots&0&0&*&\cdots&*\\
		*&*&1&\cdots&0&0&*&\cdots&*\\
		\vdots&\vdots&\vdots&\ddots&\vdots&\vdots&*&\cdots&*\\
		*&*&*&\cdots&1&0&*&\cdots&*\\
		*&*&*&\cdots&*&1&*&\cdots&*
	\end{pmatrix}_{n\times {l}}.
\end{equation*}
Thus the first $n$ columns of the matrix $E'\in\mb{R}^{n\times {l}}$ are linearly independent. Since $n\leq {l}$,  ${\rm rank}(E')=n$.  Hence $\mc{T}_{[S]}$ is strongly completely positive. 
\end{proof}
In the above proof, we provided a strongly CP decomposition \eqref{eqn:mre} of GCD tensors. We now provide an algorithm to derive such a decomposition.
\begin{algorithm}[H]
\scriptsize
\caption{Strongly CP decomposition of a GCD tensor}
\begin{algorithmic}[1]
\State Enter a set of $n$ distinct positive integers $S=\{s_{1},{s_{2},}\ldots,s_{n}\}$.
\State Construct a factor-closed set $F=\{f_{1},{f_{2},}\ldots,f_{l}\}$ containing $S.$
\State Construct an $n\times l$ zero matrix $E=(e_{ij})$.
\For {$i=1:n$}
\For {$j=1:l$}
\If{$f_{j}|s_{i}$}
\State  $e_{ij}=1$;
\Else
\State $e_{ij}=0$;
\EndIf
\EndFor
\EndFor
\State $E_{k}:=$ $k$th column of the matrix $E;$
\State \Return $\mc{T}_{[S]}=\sum_{k=1}^{l}\Phi(f_{k})(E_k)^{\otimes m}.$
\end{algorithmic}
\end{algorithm}

 In 2014,  Qi--Xu--Xu \cite{qi2014} showed that all the $H$-eigenvalues of a strongly CP tensor are always positive.  As an immediate consequence of Theorem \ref{thm:gcdcpt}, we have the following corollary.
\begin{cor}\label{cor:hevgcdt}
	All the $H$-eigenvalues of a GCD tensor are always positive.
\end{cor}
Notice that the strongly completely positive decomposition \eqref{eqn:mre} of $\mc{T}_{[S]}$ depends on a factor-closed set containing $S$. We next show that such a decomposition can be obtained using a GCD-closed set $F$ containing $S$ and the generalized Euler's totient function on $F$. To proceed further, we introduce the following notation.

\begin{definition}[\cite{bhat1991}]
Let $S=\{s_{1},s_{2},\ldots,s_{n}\}$ be a set of positive integers. 
	\begin{itemize}
		\item [(i)] $S$ is said to be \textit{greatest common divisor closed} (GCD-closed) if whenever $s_{i}$ and $s_{j}$ are in $S$, their GCD is also in $S$.
		\item [(ii)] The {\it generalized Euler's totient function} on $S$ is defined inductively as
		\begin{equation*}
			\Psi_{S}(s_{j})=s_{j}-\sum_{s_{i}|s_{j},s_{i}\neq s_{j}}\Psi_{S}(s_{i}),
		\end{equation*}
		where the empty summation is taken to be zero.
	\end{itemize}
\end{definition} 
Here are some straightforward observations from the above definition.
\begin{rem}
	\begin{itemize}
		\item [(i)] Every factor-closed set is a GCD-closed set, but the converse is not true. 
		\item [(ii)] The generalized Euler's totient function depends on the ground set $S$. If $S$ is a factor-closed set, by Lemma \ref{gauss}, $\Psi_S (s_i)=\Phi (s_i)$ for all $s_i \in S$.
	\end{itemize}
\end{rem}

Using these preliminaries, we now derive another strongly completely positive decomposition of $\mc{T}_{[S]}$ in terms of the generalized Euler's totient function.

\begin{theorem}
	Let $S=\{s_{1},s_{2},\ldots,s_{n}\}$ be a set of distinct positive integers. Let $F=\{f_{1},f_{2},\ldots,f_{l}\}$ be a GCD-closed set containing $S$. Then, the $m$th-order $n$-dimensional GCD tensor $\mc{T}_{[S]}$ can be expressed as
	\begin{equation}\label{eqn:mr}
		\mc{T}_{[S]}=\sum_{k=1}^{l}\Psi_{F}(f_{k})(E_{k})^{\otimes m},
	\end{equation}
	where $\Psi_{F}(\cdot)$ is the generalized Euler's totient function on $F$, and $E_{k}$ denotes the $k$th column of the $n\times l$ matrix $E=(e_{ij})$ such that    \begin{equation*}
		e_{ij}=\begin{cases}
			1,&\mbox{if }f_{j}\mbox{ divides } s_{i},\\
			0,&\mathrm{otherwise.}
		\end{cases}
	\end{equation*} 
\end{theorem}
The proof is similar to that of Theorem \ref{thm:gcdcpt} and is left to the interested reader.

\bigskip
We next turn our attention to the Schur (Hadamard) product of GCD tensors. For any two tensors $\mc{A}=(a_{i_{1}i_{2}\ldots{i_{m}}}), \mc{B}=(b_{i_{1}i_{2}\ldots{i_{m}}})\in\mb{R}^{[m,n]}$, their Schur product, denoted by $\mc{A}\circ\mc{B}$, is a tensor in $\mb{R}^{[m,n]}$ such that 
\begin{equation*}
	(\mc{A}\circ\mc{B})_{i_{1}i_{2}\ldots{i_{m}}}:=a_{i_{1}i_{2}\ldots{i_{m}}}\cdot b_{i_{1}i_{2}\ldots{i_{m}}}.
\end{equation*}
It is well known that the Schur product of two positive semidefinite matrices is positive semidefinite. However, Qi--Luo \cite{qi2017} in 2017 showed that the Schur product of two positive semidefinite tensors of order $m\geq4$ does not necessarily preserve positive definiteness (see Example 5.7 in \cite{qi2017}). In this sequel, a natural question is, what if we consider the Hadamard power of a positive semidefinite tensor? Unlike the case of matrices, the positive definite tensors of order $m\geq 4$ may fail to be positive definite. This can be seen in the following example.
\begin{example}
	Let $\mc{A}=(a_{ijkl})\in\mb{R}^{[4,2]}$ be a fourth-order symmetric tensor with elements:
	\begin{align*}
		\mc{A}(:,:,1,1)&=\begin{pmatrix}
			3 & -2\\
			-2 & 1
		\end{pmatrix},&\mc{A}(:,:,1,2)&=\begin{pmatrix}
			-2  & 1\\
			1  & 1
		\end{pmatrix},\\
		\mc{A}(:,:,2,1)&=\begin{pmatrix}
			-2 & 1\\
			1 & 1
		\end{pmatrix},&\mc{A}(:,:,2,2)&=\begin{pmatrix}
			1 & 1\\
			1 & 1
		\end{pmatrix}.
	\end{align*}
	The $H$-eigenvalues of the tensor $\mc{A}$ are $0.5013, 0.9798, 3.8065, 8.3941.$ The Hadamard square of $\mc{A}$, denoted as $\mc{A}^{\circ 2}$, is obtained by squaring each element of $\mc{A}$ entrywise:
	\begin{align*}
		\mc{A}^{\circ 2}(:,:,1,1)&=\begin{pmatrix}
			9 & 4\\
			4 & 1
		\end{pmatrix},&\mc{A}^{\circ 2}(:,:,1,2)&=\begin{pmatrix}
			4  & 1\\
			1  & 1
		\end{pmatrix},\\
		\mc{A}^{\circ 2}(:,:,2,1)&=\begin{pmatrix}
			4 & 1\\
			1 & 1
		\end{pmatrix},&\mc{A}^{\circ 2}(:,:,2,2)&=\begin{pmatrix}
			1 & 1\\
			1 & 1
		\end{pmatrix}.
	\end{align*}
	The $H$-eigenvalues of $ \mc{A}^{\circ 2}$ are $-2.1138$ and $  20.0391.$ Since the $H$-eigenvalues of a positive semidefinite tensor are always nonnegative, $\mc{A}^{\circ 2}$ is not positive semidefinite. The $H$-eigenvalues in this example are obtained using the {\rm MATLAB} package \texttt{TenEig} $2.0$, developed by Chen--Han--Zhou \cite{chen2016}.
	\end{example}
	In 2016, Luo--Qi \cite{luo2016} showed that the classes of CP tensors and strongly CP tensors are closed under the Hadamard product. Since GCD tensors are strongly CP, we have the following corollary.
	\begin{cor}\label{cor:cphp}
	Let $U=\{u_{1},u_{2},\ldots,u_{n}\}$ and $V=\{v_{1},v_{2},\ldots,v_{n}\}$ be the sets of distinct positive integers, and $\mc{T}_{[U]}$, $\mc{T}_{[V]}$ be the $m$th-order $n$-dimensional GCD tensors on the sets $U$ and $V$ respectively. Then $\mc{T}_{[U]}\circ\mc{T}_{[V]},$ is strongly CP. 
\end{cor}

The above result immediately implies that for any given set $S=\{s_{1},s_{2},\ldots,s_{n}\}$ of distinct positive integers, the positive integer Hadamard power of the $m$th-order GCD tensor $\mc{T}_{[U]}$ is strongly CP. It is natural to ask {whether the}  {positive} fractional powers of the $m$th-order GCD tensor have the same property.  Our second main result \ref{thrmfractpwr} provides a positive answer to this question. In fact we prove a stronger result. To continue our discussion, we need three basic definitions.
\begin{definition}
	\begin{enumerate}[(i)]
    \item Let $S\subseteq \mb{R}$ and $f: S\rightarrow \mb{R}$ be a function. We define $f[\mc{A}]:=(f(a_{i_{1}{i_{2}}\ldots i_{m}}))$ for all $\mc{A}=(a_{i_{1}{i_{2}}\ldots i_{m}})\in \mb{R}^{[m,n]}$ with $a_{i_{1}{i_{2}}\ldots i_{m}} \in S$.
		\item A function $f:\mb{N}\to \mb{R}$ is said to be multiplicative if $f(xy)=f(x)f(y)$ whenever $\GCD(x,y)=1$.
		\item The Dirichlet convolution of two multiplicative functions $f$ and $g$ is denoted and defined as
		\[(f*g)(k):=\sum\limits_{d|k}f(d)g\left(\frac{k}{d}\right).\]
		\item [(iii)] The {\it M\"obius function} $\mu:\mathbb{N}\to \{0,\pm 1\}$ is defined by 
		\begin{equation*}
			\mu(k)=\begin{cases}
				1,&\text{ if }n=1,\\
				(-1)^{l},&\text{ if $k$ is the product $l$ distinct primes for $l\geq 1$},\\
				0,&\text{ otherwise}.
			\end{cases}
		\end{equation*}
	\end{enumerate}
\end{definition}
Clearly, the M\"obius function is multiplicative. It is well known that for any multiplicative function $g$, $\sum\limits_{d|p}(g*\mu)(d)=g(p).$

We now prove that for any set $S=\{s_1,{s_{2},}\ldots,s_n\}$ of distinct positive integers and for any multiplicative map  $g$ with all $(g*\mu)(p)>0$, the tensor $(g(\GCD(s_{i_{1}},s_{i_{2}},\ldots,s_{i_{m}})))\in \mb{R}^{[m,n]}$ is strongly CP.

\begin{theorem}\label{thm:gcdmult}
Let $S=\{s_{1},s_{2},\ldots, s_{n}\}$ be a set of distinct positive integers and  $\mc{T}_{[S]}=(t_{i_{1}i_{2}\ldots i_{m}})$ be an $m$th-order $n$-dimensional GCD tensor on $S$. Let $g:\mathbb{N}\to \mathbb{R}$ be a multiplicative map with $(g*\mu)(p)>0$ for all $p\in \mb{N}$. Then the tensor $g[\mc{T}_{[S]}]=(g(t_{i_{1}i_{2}\ldots i_{m}}))\in \mb{R}^{[m,n]}$ is strongly CP.
\end{theorem}
\begin{proof}
	Let $F=\{f_{1},f_{2},\ldots,f_{l}\}$ be a factor-closed set containing $S.$ Define the matrix $E=[E_1,{E_{2},}\ldots,E_l]\in\mb{R}^{n\times l}$ as in \eqref{matrixE}, where $E_1,{E_{2},}\ldots,E_l$ are the columns of $E$. Then ${\rm span}\{E_1,{E_{2},}\ldots,E_l\}=\mb{R}^n$. Since $(g*\mu)(p)>0$ for all $p \in \mb{N}$, to complete the proof it suffices to show that 
	\[g[\mc{T}_{[S]}]=\sum_{k=1}^{l}(g*\mu)(f_{k})(E_{k})^{\otimes m}.\]

Since $\sum\limits_{d|p}(g*\mu)(d)=g(p)$, by a similar argument as in the proof of Theorem \ref{thm:gcdcpt}, we have \[\left(\sum_{k=1}^{l}(g*\mu)(f_{k})(E_{k})^{\otimes m}\right)_{i_{1}{i_{2}}\ldots i_{m}}=g(\GCD(s_{i_{1}},s_{i_{2}},\ldots,s_{i_{m}})),\]
for all $1\leq i_1,{i_{2},}\ldots,i_m\leq n$. Thus $g[\mc{T}_{[S]}]$ is strongly CP.
\end{proof}
As an application of the above result, we now have our proof of the second main result.
\begin{proof}[Proof of Theorem \ref{thrmfractpwr}]
Let $r$ be a positive real number and define $g(p):=p^r$ for all $p \in \mb{N}$. Then $g$ is multiplicative and $(g*\mu)(p)>0$ for all $p \in \mb{N}$. By Theorem \ref{thm:gcdmult}, $g[\mc{T}_{[S]}]$ is strongly CP. Thus $\mc{T}^{\circ r}_{[S]}$ is strongly CP for all $r>0$.
\end{proof}
\section{Determinant of GCD the Tensor}\label{sec:det}
The main goal of this section is to prove Theorem 1.8. This requires the notion of the general product of two tensors and some preliminary results.
In 2013, Shao \cite{shao2013} introduced the notion of the general product (or composite) of two $n$-dimensional tensors.
\begin{definition}\cite[Definition 1.1]{shao2013}
	Let $\mc{A}\in\mb{R}^{[m,n]}$ and $\mc{B}\in\mb{R}^{[k,n]}$. The general product $\mc{A}\cdot\mc{B}$ of $\mc{A}$ and $\mc{B}$ is an $((m-1)(k-1)+1)$th-order $n$-dimensional tensor, and is defined elementwise:
	\begin{eqnarray*}
		(\mc{A}\cdot\mc{B})_{i_{1}\alpha_{1}\ldots \alpha_{(m-1)}}=\sum_{j_{2},\ldots,j_{m}=1}^{n}a_{i_{1}j_{2}\ldots j_{m}}b_{j_{2}\alpha_{1}}\cdots b_{j_{m}\alpha_{(m-1)}},
	\end{eqnarray*}
	where $i_{1}\in[n],\alpha_{1},\ldots,\alpha_{m-1}\in[n]^{k-1}=[n]\times\cdots \times[n]$ and $[n]=\{1,2,\ldots,n\}$.
\end{definition}
The general product of tensors satisfies many interesting properties, including: (i) The associative law. (ii) When the two tensors are matrices (or one is a matrix and the other is a vector), the general tensor product coincides with the usual matrix product.

We now recall four preliminary results involving the general product which are crucial to prove Theorem \ref{detgcd}.  The first result is a formula of Shao--Shan--Zhang to derive the determinant of the general product of two tensors. 
\begin{theorem}\cite[Theorem 2.1]{shao2013a}\label{thmdet2}
	Let $\mc{A}\in\mb{R}^{[m,n]}$ and $\mc{B}\in\mb{R}^{[k,n]},$ where $m,k\geq2.$ Then,
	\[\det(\mc{A}\cdot\mc{B})=\det(\mc{A})^{(k-1)^{(n-1)}}\det(\mc{B})^{(m-1)^{n}}.\]
\end{theorem} 

\begin{lemma}\cite[Lemma 2.1]{shao2013}\label{lemdet2}
Let $\mc{A}\in\mb{R}^{[m,n]}$ and $Q_1, Q_2 \in \mathbb{R}^{n \times n}$. Then
\[\det (Q_1\cdot \mc{A} \cdot Q_2)= \det (Q_1\cdot \mc{I} \cdot Q_2) \det (\mc{A}),\]
where $\mc{I}\in\mb{R}^{[m,n]}$ is the identity tensor. 
\end{lemma}
The next lemma gives a formula to compute the determinant of a diagonal tensor.
\begin{lemma}(\cite[Proposition 5.1]{hu2013}, \cite[Example 3.3]{shao2013a})\label{lemdet3}
Let $\mc{A}$ be the $m$th-order $n$-dimensional diagonal tensor with the $i$th diagonal element of $\mc{A}$ is $a_{ii\ldots i} = d_{i}$  for all $i = 1,{2,}\ldots,n$. Then \[\det(\mc{A})=\prod_{i=1}^{n}(d_{i})^{(m-1)^{(n-1)}}.\]
\end{lemma}
The next lemma gives a relationship between the $k$-mode product of a tensor $\mc{A}$ with a matrix $Q$ and the general product of $\mc{A}$ with $Q$. 
\begin{lemma}\cite[Remark 2.3]{shao2013}\label{lemdet1}
	Let $\mc{A}\in\mb{R}^{[m,n]}$ and $Q\in\mb{R}^{n \times n}$. Then
	\[\mc{A}\times_{1}Q\times_{2} Q\times_{3}\cdots\times_{m} Q= Q\cdot\mc{A}\cdot Q^{T}.\]
\end{lemma}
Let $\mc{A}, \mc{B}\in\mb{R}^{[m,n]}$. Then $\mc{A}$ is congruent to $\mc{B}$ if there exists $Q\in\mb{R}^{n \times n}$ such that $\mc{B}=Q\cdot\mc{A}\cdot Q^{T}$.  If $Q$ is a permutation matrix, then we say $\mc{A}$ is permutationally congruent to $\mc{B}$.
\begin{theorem}\label{thmsmlr}
	Let $S=\{s_{1},{s_{2},}\ldots,s_{n}\}$ be an ordered set of distinct positive integers and let $S'=\{s_{i_{1}},{s_{i_{2}},}\ldots,s_{i_{n}}\}$ be a rearrangement of the elements of $S$. Let $\mc{T}_{[S]}$ and $\mc{T}_{[S']}$ be the $m$th-order $n$-dimensional GCD tensors on $S$ and $S'$ respectively. Then $\mc{T}_{[S']}$ is permutationally congruent to $\mc{T}_{[S]}.$ Moreover, $\det{(\mc{T}_{[S']})}=\det{(\mc{T}_{[S]})}.$
\end{theorem}
\begin{proof}
	Let $\sigma$ be the permutation on the set $\{1,2,\ldots,n\}$ such that $S^\prime=\{s_{\sigma({1})},{s_{\sigma(2)},}\ldots,s_{\sigma({n})}\}$  and let $P=P_{\sigma}\in\mb{R}^{n\times n}$ be the corresponding  permutation matrix. We claim that $\mc{T}_{[S']}=P\cdot \mc{T}_{[S]} \cdot P^T$. Let $t_{j_{1}{j_{2}}\ldots j_{m}}:=\GCD(s_{j_{1}},s_{j_{2}},\ldots,s_{j_{m}})$ for all $1\leq j_1,{j_{2},}\ldots,j_m\leq n$. Now notice that for all $1\leq i_1,{i_{2},}\ldots,i_m\leq n$,
	\begin{eqnarray*}
		(\mc{T}_{[S]}\times_{1}P\times_{2}P\times_{3}\cdots\times_{m}P)_{i_{1}\ldots i_{m}}&=&\sum_{j_{1},{j_{2},}\ldots,j_{m}=1}^{n}t_{j_{1}{j_{2}}\ldots j_{m}}p
		_{i_{1}j_{1}}{\cdot p
		_{i_{2}j_{2}}}\cdots p_{i_{m}j_{m}}\\
		&=& t_{\sigma(i_{1}){\sigma(i_{2})}\ldots \sigma(i_{m})} ~(\because p_{ij}=1\Leftrightarrow j=\sigma(i))\\
		&=&\GCD(s_{\sigma(i_{1})},{s_{\sigma(i_{2})},}\ldots,s_{\sigma(i_{m})})\\
		&=&(\mc{T}_{[S']})_{i_{1}{i_{2}}\ldots i_{m}}.
	\end{eqnarray*}
	Thus, $\mc{T}_{[S']}=\mc{T}_{[S]}\times_{1}P\times_{2}P\times_{3}\cdots\times_{m}P$. By Lemma \ref{lemdet1}, $\mc{T}_{[S']}= P\cdot \mc{T}_{[S]} \cdot P^T$.
	
	We next prove the second part of the assertion. By Lemma \ref{lemdet2}, 
	\[\det(\mc{T}_{[S']})=\det(P\cdot \mc{I}\cdot P^{T})\det(\mc{T}_{[S]}).\]
	
	Since $P$ is a permutation matrix, $P\cdot \mc{I}\cdot P^{T}=\mc{I}$. Thus $\det(P\cdot \mc{I}\cdot P^{T})=\det(\mc{I})=1$ and $\det(\mc{T}_{[S']})=\det(\mc{T}_{[S]})$.
\end{proof}

The next lemma gives an important factorization of the GCD tensor, which is crucial to obtain the determinant.

\begin{lemma}\label{factodet}
	Let $S=\{s_{1},s_{2},\ldots,s_{n}\}$ be a set of $n$ distinct positive integers, and $F=\{f_{1},{f_{2},}\ldots,f_{l}\}$ be a factor-closed set containing $S$. Then the $m$th-order $n$-dimensional GCD tensor $\mc{T}_{[S]}$ can be expressed as \begin{equation}\label{deteq}
	\mc{T}_{[S]}=\mc{D}\times_{1}E\times_{2}E\times_{3}\cdots\times_{m}E,
\end{equation} where $\mc{D}\in \mb{R}^{[m,l]}$ is a diagonal tensor with diagonal entries $d_{ii\ldots i}=\Phi(f_{i}),$ $1\leq i\leq l,$ and $E=(e_{ij})\in\mb{R}^{n\times l}$ such that $e_{ij}=1$ if $f_{j}$ divides $s_{i},$ otherwise $e_{ij}=0.$
\end{lemma}
\begin{proof}
	To prove the identity \eqref{deteq} it suffices to show that for all $1\leq i_1,{i_2,}\ldots,i_m\leq n$, the $(i_{1},{i_{2},}\ldots, i_{m})$-th element of $\mc{D}\times_{1}E\times_{2}E\times_{3}\cdots\times_{m}E$ equals the $(i_{1},{i_{2},}\ldots, i_{m})$-th element of $\mc{T}_{[S]}$.  Now
	\begin{eqnarray*}
(\mc{D}\times_{1}E\times_{2}E\times_{3}\cdots\times_{m}E)_{i_{1}{i_{2}}\ldots i_{m}}
		&=&\sum_{j_{1},{j_{2},}\ldots, j_{m}=1}^{l}d_{j_{1}{j_{2}}\ldots j_{m}}e_{i_{1}j_{1}}e_{i_{2}j_{2}}\ldots e_{i_{m}j_{m}}\\
		&=&\sum_{f_{k}\mid s_{i_{1}},f_{k}\mid s_{i_{2}},\ldots,f_{k}\mid s_{i_{m}}}\Phi(f_{k})\\
		&=&\sum_{f_{k}\mid \GCD (s_{i_{1}},s_{i_{2}},\ldots, s_{i_{m}})}\Phi(f_{k})\\
		&=&\GCD (s_{i_{1}},s_{i_{2}},\ldots, s_{i_{m}})\\
		&=&(\mc{T}_{[S]})_{i_{1}i_{2}\ldots i_{m}}.
	\end{eqnarray*}
	Thus $\mc{T}_{[S]}=\mc{D}\times_{1}E\times_{2}E\times_{3}\cdots\times_{m}E$.
\end{proof}
Based on Lemma \ref{factodet}, we now provide an algorithm to factorize the GCD tensor as in \eqref{deteq}.
\begin{algorithm}[H]
\scriptsize
\caption{Factorization of GCD tensor}
\begin{algorithmic}[1]
\State Enter a set of $n$ distinct positive integers $S=\{s_{1},{s_{2},}\ldots,s_{n}\}$.
\State Construct a factor-closed set $F=\{f_{1},{f_{2},}\ldots,f_{l}\}$ containing $S.$
\State Construct an $n\times l$ zero matrix $E=(e_{ij})$.
\For {$i=1:n$}
\For {$j=1:l$}
\If{$f_{j}|s_{i}$}
\State  $e_{ij}=1$;
\Else
\State $e_{ij}=0$;
\EndIf
\EndFor
\EndFor
\State Construct an $m$th-order $n$-dimensional zero tensor $\mc{D}$;
\For {$i = 1:n$} 
\State ${d_{ii\ldots i}}=\Phi(f_{i})$;
\EndFor
\State Compute $\mc{T}_{[S]}=\mc{D}\times_{1}E\times_{2}\cdots\times_{m}E.$
\State \Return $\mc{D},E$ and $\mc{T}_{[S]}.$
        \end{algorithmic}
    \end{algorithm}
With these ingredients in hand, we are now ready to prove our final main result.
\begin{proof}[Proof of Theorem \ref{detgcd}]
	Using Theorem \ref{thmsmlr}, without loss of generality we may assume $S=\{s_{1},s_{2},\ldots,s_{n}\}$ with $s_{1}<s_{2}<\cdots<s_{n}$. By Lemma \ref{factodet}, there exist a diagonal tensor $\mc{D}\in \mb{R}^{[m,n]}$ with diagonal entries with $d_{ii\ldots i}=\Phi(s_{i}),$ $1\leq i\leq n,$ and a square matrix $E=(e_{ij})\in\mb{R}^{n\times n}$ such that
	
	\[\mc{T}_{[S]}=\mc{D}\times_{1}E\times_{2}E\times_{3}\cdots\times_{m}E,\]
	where $e_{ij}=1$ if $s_{j}$ divides $s_{i},$ otherwise $e_{ij}=0.$ 
Since $s_{1}<s_{2}<\cdots<s_{n}$, $E$ is a lower triangular matrix with diagonal entries are $1$. Thus $\det (E)=1$.  By Lemma \ref{lemdet1}, $\mc{T}_{[S]}=E\cdot\mc{D}\cdot E^T$, where `$\cdot$' denotes the general product of tensors.  Since the general product of tensors is associative, by Theorem \ref{thmdet2}, we have
	\begin{eqnarray*}
		\det(\mc{T}_{[S]})&=&\det(E\cdot\mc{D}\cdot E^T)\\
		&=&\det(E\cdot \mc{D})\det(E^T)^{(m-1)^{n}}\\
		&=&\det(E)^{(m-1)^{(n-1)}} \det(\mc{D})\det(E)^{(m-1)^{n}}\\
		&=&\det(\mc{D})\\
		&=&\prod_{i=1}^{n}\Phi(s_{i})^{(m-1)^{(n-1)}}.
	\end{eqnarray*}
\end{proof}
Since GCD tensor is defined for any set $S$ of distinct positive integers, it is natural to investigate what happens if we remove the factor-closed restriction in Theorem \ref{detgcd}. Based on numerical evidence, we propose the following conjecture.
\begin{conjecture}
    Let $S=\{s_{1},{s_{2},}\ldots,s_{n}\}$ be an ordered set of distinct positive integers and $\mc{T}_{[S]}$ be the $m$th-order n-dimensional GCD tensor on $S.$ Then, 
    \begin{equation*}
        \det{(\mc{T}_{[S]})}\geq \prod_{i=1}^{n}\Phi(s_{i})^{(m-1)^{(n-1)}},
    \end{equation*}
    and equality holds if and only if $S$ is a factor-closed set.
\end{conjecture}

Now we switch tracks. For a factor-closed set $S$, we now compute the determinant of $g[\mc{T}_{[S]}],$ where $g$ is a multiplicative map. 
\begin{theorem}\label{gcdmult}
	Let $S=\{s_{1},s_{2},\ldots, s_{n}\}$ be a factor-closed set of distinct positive integers and $\mc{T}_{[S]}$ be the $m$th-order n-dimensional GCD tensor on $S.$ Let  $g:\mathbb{N}\to \mathbb{R}$ be a multiplicative map. 
	Then
	\[\det (g[\mc{T}_{[S]}])= \prod_{i=1}^{n}(g*\mu)(s_{i})^{(m-1)^{(n-1)}}.\]
\end{theorem}
\begin{proof}
By an argument similar to  the proof of  Theorem \ref{thmsmlr}, $\det(g[\mc{T}_{[S]}])=\det(g[\mc{T}_{[S^\prime]}])$ for any rearrangement $S^\prime$ of $S$.  Thus,  without loss of generality  assume that $S=\{s_{1},s_{2},\ldots,s_{n}\}$ with $s_{1}<s_{2}<\cdots<s_{n}$. Again, by a similar argument as in the proof of Theorem \ref{detgcd}, we have
\[ g[\mc{T}_{[S]}]=\mc{D}_g\times_{1}E\times_{2}E\times_{3}\cdots\times_{m}E,\]
where  $E \in \mathbb{R}^{n \times n}$ is defined as in the proof of Theorem \ref{detgcd} and $\mc{D}_g\in \mb{R}^{[m,n]}$ is a diagonal tensor with diagonal entries $d_{ii\ldots i}={(}g*\mu{)}(s_{i}),$ $1\leq i\leq n.$ Since $\det(E)=1$, by Theorem \ref{thmdet2}, 

\begin{eqnarray*}
\det (g[\mc{T}_{[S]}]) &=& \det (\mc{D}_g)\\
&=&  \prod_{i=1}^{n}(g*\mu)(s_{i})^{(m-1)^{(n-1)}}.
\end{eqnarray*}
\end{proof}
In the final result of the section, we obtain the determinant of $\mc{T}_{[S]}$ for a GCD-closed set $S$. 
\begin{theorem}
	Let $S=\{s_{1},{s_{2},}\ldots,s_{n}\}$ be GCD-closed set of distinct positive integers. Let $\mc{T}_{[S]}$ be the $m$th-order $n$-dimensional GCD tensor on $S$.  Then \[\det{(\mc{T}_{[S]})}=\prod_{i=1}^{n} \Psi_{S}(s_{i})^{(m-1)^{(n-1)}},\] where $\Psi_{S}$ is the generalized Euler's totient function on $S$.
\end{theorem}
\begin{proof}
By Theorem \ref{thmsmlr}, without loss of generality we may assume $S=\{s_{1},s_{2},\ldots,s_{n}\}$ with $s_{1}<s_{2}<\cdots<s_{n}$. Then, the matrix $E=(e_{ij})$ such that $e_{ij}=1$ if $s_{j}$ divides $s_{i}$ and $e_{ij}=0$ otherwise becomes an $n\times n$ lower triangular matrix with diagonal entries $1$. Thus, $\det(E)=1.$ Construct a diagonal tensor $\mc{D}\in \mb{R}^{[m,n]}$ with $d_{ii\ldots i}=\Psi_{S}(s_{i}),$ $1\leq i\leq n.$ Then, one can verify that $$\mc{T}_{[S]}=\mc{D}\times_{1}E\times_{2}E\times_{3}\cdots\times_{m}E.$$ 
	Using Lemma \ref{lemdet3}, we have $\det{(\mc{D})}=\prod\limits_{i=1}^{n}\Psi_{S}(s_{i})^{(m-1)^{(n-1)}}$. Then 
	\begin{eqnarray*}
		\det(\mc{T}_{[S]})&=&\det(\mc{D}\times_{1}E\times_{2}E\times_{3}\cdots\times_{m}E)\\
		&=&\det(E\cdot \mc{D}\cdot E^{T}),\mbox{where} \cdot \mbox{denotes the general product of tensors}\\
		&=&\det(\mc{D})\det(E)^{(m-1)^{(n-1)}m}\\
		&=&\det(\mc{D})\\&=&\prod_{i=1}^{n}\Psi_{S}(s_{i})^{(m-1)^{(n-1)}}.
	\end{eqnarray*}
\end{proof}

\section{Generalization of GCD tensor on meet semi-lattice}\label{sec:poset}
In this section, we generalize the notion of GCD tensors to a general partially ordered set (poset) and call it as {\it meet tensor}. Let us begin with some basic definitions.
\begin{definition}
        A poset $(U,p)$ with partial order $p$ is called meet semi-lattice if for all $u_{1},u_{2}\in U$, there exists a unique $w\in U$ such that
        \begin{enumerate}[(i)]
            \item $wpu_{1}, wpu_{2},$ and 
            \item if $zpu_{1}, zpu_{2}$ for some $z\in U,$ then $zpw.$
        \end{enumerate}
        Here $w$ is called meet of $u_{1},u_{2}$ and is denoted by $(u_{1},u_{2})_{p}.$
    \end{definition}
\begin{definition}
       Let $(U,p)$ be a meet semi-lattice and $S=\{s_{1},{s_{2},}\ldots,s_{n}\}$ be a subset of $U.$ Let $g$ be any real valued function on $U.$ Then the $m$th-order $n$-dimensional tensor $\mc{T}_{[S]}^{g}=(t_{i_{1}i_{2}\ldots i_{m}})\in\mb{R}^{[m,n]},$ where 
        \begin{equation*}
            t_{i_{1}i_{2}\ldots i_{m}}=g((s_{i_{1}},s_{i_{2}},\ldots,s_{i_{m}})_{p}), ~1\leq i_{k}\leq n, \mbox{ for } k=1,2,\ldots,m,
        \end{equation*}
        is called the {\it meet tensor} on $S$ with respect to $g$.
\end{definition}
To discuss the structure theorem and the determinant of a meet tensor on a generalized poset, we require the following two definitions.
    \begin{definition}\cite[Definition 7]{bhat1991}
Let $(U,p)$ be a meet semi-lattice and $S$ be a subset of $U.$
\begin{enumerate}[(i)]
	\item The set $S$ is called {\it meet-closed} if for every $s_{1},s_{2}\in S$ their meet $(s_{1},s_{2})_{p}$ is also in $S$.
	
	\item Let $g$ be a real valued function on $U.$ Then the generalized Euler's totient function $\Psi_{S,g}$ on $S$ with respect to $g$  is given by
	\begin{equation*}
		\Psi_{S,g}(s_{i})=g(s_{j})-\sum_{\substack{x_{j}px_{i}\\ x_{i}\neq x_{j}}}\Psi_{S,g}(s_{j}).
	\end{equation*}
	The empty summation is taken to be zero. 
\end{enumerate}
    \end{definition}
    
    We now define a matrix which will be crucial for this section. Let $(U,p)$ be a meet semi-lattice and $S=\{s_{1},{s_{2},}\ldots,s_{n}\}$ be a subset of $U$, and  $F=\{f_{1},f_{2},\ldots,f_{l}\}$ be a meet-closed finite subset of $U$ containing $S$. Define an $n\times l$ matrix $E:=(e_{ij})$  such that
        \begin{equation*}
            e_{ij}:=\begin{cases}
                1,&\text{if }f_{j}ps_{i},\\
                0,&\text{otherwise.}
            \end{cases}
        \end{equation*}
With these basic definitions, we now extend the results of GCD tensor to meet tensor. The first result of this section gives a structure of the meet tensor.
\begin{theorem}\label{thm:mtscp}
    Let $(U,p)$ be a meet semi-lattice and $S=\{s_{1},s_{2},\ldots,s_{n}\}$ be a subset of $U.$ Let $g$ be any real valued function on $U$ and $F=\{f_{1},f_{2},\ldots,f_{l}\}$ be a meet-closed finite subset of $U$ containing $S$. Then the $m$th-order $n$-dimensional meet tensor $\mc{T}_{[S]}^{g}$ can be decomposed as
    \begin{equation*}
        \mc{T}_{[S]}^{g}=\sum_{k=1}^{l}\Psi_{F,g}(f_{k})(E_{k})^{\otimes m}.
    \end{equation*}
   where $E_{k}$ is the $k$th column of the matrix $E.$
\end{theorem}

Next, we turn our attention to the determinant of the meet tensor. Like the case of GCD tensors on a set of positive integers, the meet tensor on a meet semi-lattice also factors as Lemma \ref{factodet}.
\begin{lemma}\label{lma:factmtnsr}
        Let $(U,p)$ be a meet semi-lattice, $g$ be any real valued function on $U$, and $S=\{s_{1},s_{2},\ldots,s_{n}\}$ be a subset of $U.$ Let $F=\{f_{1},f_{2},\ldots,f_{l}\}$ be a meet-closed finite subset of $U$ containing $S$. Then the meet tensor $\mc{T}_{[S]}^{g}$ can be expressed as
        \begin{equation*}
            \mc{T}_{[S]}^{g}=\mc{D}\times_{1} E\times_{2} E\times_{3} \cdots \times_{m} E,
        \end{equation*}
        where $\mc{D}$ is an $m$th-order $l$-dimensional  diagonal tensor with its diagonal entry $d_{ii\ldots i}=\Psi_{F,g}(f_{i}).$
    \end{lemma}
Using Lemma \ref{lma:factmtnsr}, we now derive the determinant of a meet tensor on a meet-closed set.
\begin{theorem}\label{thm:detmtnsr}
    Let $(U,p)$ be a meet semi-lattice and $S=\{s_{1},{s_{2},}\ldots,s_{n}\}$ be a meet-closed subset of $U$. Let $g$ be any real valued function on $U$ and let  $\mc{T}_{[S]}^{g}$ be the $m$th-order $n$-dimensional meet tensor with respect to $g$. Then
    \begin{equation*}
        \det(\mc{T}_{[S]}^{g})=\prod_{i=1}^{n}\Psi_{S,g}(s_{i})^{(m-1)^{(n-1)}}.
    \end{equation*}
\end{theorem}
We omit the proofs of Theorem \ref{thm:mtscp}, Lemma \ref{lma:factmtnsr} and Theorem \ref{thm:detmtnsr} as they are verbatim to the proofs for GCD tensors.
\section*{Acknowledgements}
 We also thank the anonymous referee for providing useful comments and a reference that improved the manuscript. P.N. Choudhury was partially supported by Prime Minister Early Career Research Grant (PMECRG) ANRF/ECRG/2024/002674/PMS (ANRF, Govt. of India), INSPIRE Faculty Fellowship research grant DST/INSPIRE/04/2021/002620 (DST, Govt. of India), and IIT Gandhinagar Internal Project: IP/IP/50025. K. Panigrahy was supported by IIT Gandhinagar Post-Doctoral Fellowship IP/IP/50020.
\bibliographystyle{plain}
\bibliography{ref}
\end{document}